\documentclass{amsart}[12pt]

\usepackage{amsmath}
\usepackage{amsfonts}
\usepackage{amssymb}
\usepackage{cite}
\usepackage[usenames]{color}
\usepackage{amsxtra}
\usepackage{amscd}
\usepackage{color,xcolor}

\newtheorem{theorem}{Theorem}
\newtheorem{lemma}{Lemma}

\theoremstyle{definition}

\newtheorem{remark}{Remark}

\begin{document}
\title[Volterra type operators on minimal M\"{o}bius invariant space]%
{Volterra type operators on minimal M\"{o}bius invariant space}

\author{Huayou Xie, Junming Liu, Saminathan Ponnusamy}

\address{Department of Mathematics, Sun Yat-sen University, Guangzhou, 510275 P.~R.~China}\email{xiehy33@mail2.sysu.edu.cn}

\address{School of Mathematics and Statistics, Guangdong University of Technology, Guangzhou, Guangdong, 510520, P.~R.~China}\email{jmliu@gdut.edu.cn}

\address{S. Ponnusamy, Department of Mathematics, Indian Institute of Technology Madras, Chennai-600 036, India.} \email{samy@iitm.ac.in}

\begin{abstract}
In this note, we mainly study  operator-theoretic properties  on  Besov space $B_{1}$ on the unit disc. This space is the minimal M\"{o}bius invariant space.
 Firstly, we consider the boundedness of Volterra type operators.  Secondly, we prove that Volterra type operators belong to the Deddens algebra of composition operator.  Thirdly, we obtain  estimates for the essential norm of Volterra type operators. Finally, we give a complete characterization of spectrum of Volterra type operators.
\end{abstract}

\thanks{This work was supported by NNSF of China (Grant No. 11801094).}

\maketitle

\section{Introduction and preparation}\label{sec1}
In this paper, $\mathbb{D}$ denotes the open  unit disc and $\mathbb{T}$ be the unit circle.
Let $H(\mathbb{D})$ be the  space of all analytic function on $\mathbb{D}$.
For $0<p<\infty$, the Hardy space $H^p$ consists of analytic functions $f\in\mathbb{D}$ such that
$$\|f\|^p_{H^p}=\sup_{0\leq r<1}\frac{1}{2\pi}\int^{2\pi}_{0}|f(re^{i\theta})|^p\,d\theta<\infty.
$$
If $p=\infty$, then $H^{\infty}$ is the space of bounded analytic functions $f$ in $H(\mathbb{D})$ with
$$\|f\|_{\infty}=\sup\{|f(z)|:z\in \mathbb{D}\}.
$$

For $0<p<\infty$, the Bergman space $A^p$ consists of all function $f$ analytic in  $\mathbb{D}$ such that
$$\|f\|^p_{A^p}=\int_{\mathbb{D}}|f(z)|^p \,dA(z)<\infty,
$$
where $dA(z)$ is the normalized Lebesgue area measure on $\mathbb{D}$. It is clear that $H^p\subset A^p$.  Moreover,  $H^p\subset A^{2p}$ and $\|f\|_{A^{2p}}\leq \|f\|_{H^p}$ for $0<p<\infty$. See \cite{V}, for example.

The Dirichlet type space $\mathcal{D}^p$ is the set of all functions $f\in H(\mathbb{D})$ with
$$\|f\|^p_{\mathcal{D}^p}=|f(0)|^p+\int_{\mathbb{D}}|f'(z)|^p \,dA(z)<\infty.
$$

The space of all conformal automorphisms of $\mathbb{D}$ forms a group, called the M\"{o}bius group and   is denoted by ${\rm Aut}\, (\mathbb{D})$. It is well-known that $\varphi$ belongs to
$ {\rm Aut}\, (\mathbb{D})$ if and only if there exists a real number $\theta$ and a point $a\in \mathbb{D}$ such that
\begin{align*}
\varphi(z)=e^{i\theta}\sigma_{a}(z) ~\mbox{ and }~ \sigma_{a}(z)=\frac{a-z}{1-\overline{a}z},  \quad z\in\mathbb{D}.
\end{align*}
Let $X$ be a Banach  space of analytic functions on $\mathbb{D}$. Then $X$ is said to be M\"{o}bius invariant  whenever $f\circ\varphi\in X$ for all $f\in X$ and $\varphi \in {\rm Aut}\, (\mathbb{D})$  and $\|f\circ\varphi\|_{X}=\|f\|_{X}$.

For $1<p<\infty$, the Besov space $B_p$ consists of analytic functions $f$ in
$\mathbb{D}$ such that
$$\int_{\mathbb{D}}(1-|z|^2)^{p-2}|f'(z)|^{p}\,dA(z)<\infty.
$$
The norm of $B_p$ is defined by
$$\|f\|_{B_p}=|f(0)|+\left(\int_{\mathbb{D}}(1-|z|^2)^{p-2}|f'(z)|^{p}\,dA(z)\right)^{\frac{1}{p}}.
$$
When $p=\infty$, $B_{\infty}=: \mathcal{B}$ is called the classical Bloch space. We define a norm in $\mathcal{B}$ by
$$\|f\|_{\mathcal{B}}=|f(0)|+\sup_{z\in\mathbb{D}}(1-|z|^2)|f'(z)|<\infty.
$$
When $p=2$, $B_2=:\mathcal{D}$ is the classical Dirichlet space. When $p=1$, we get the analytic Besov space $B_1$ which is the minimal M\"{o}bius invariant space consisting of all function $f\in H(\mathbb{D})$ with
$$f(z)=\sum^{\infty}_{k=1}c_k\sigma_{a_k}(z),
$$
where the sequence $\{c_k\}_{k\geq 1}\in \ell^1$ and $\{a_k\}_{k\geq 1}\in \mathbb{D}$. An equivalent norm of $B_1$ is defined as
$$\|f\|_{B_1}=|f(0)|+|f'(0)|+\int_{\mathbb{D}}|f''(z)|\,dA(z).
$$
 Arazy et al. \cite{AJ}  first studied minimal M\"{o}bius invariant space systematically.  More results related to minimal M\"{o}bius invariant space may be seen from \cite{BW,B,C,M,O,W}.

Now, we define several operators on $B_1$. For $g\in H(\mathbb{D})$,
the multiplication operator $M_g$ on $B_1$ is defined by
$$(M_{g})f(z)=f(z)g(z),~ f\in B_1,~ z\in\mathbb{D}.
$$
The differentiation operator is given by $Df=f'$ for each $f\in H(\mathbb{D})$.
Given $g\in H(\mathbb{D})$, the Volterra type operator $T_g$ is defined by
$$(T_gf)(z)=\int^{z}_{0}f(w)g'(w)\,dw, ~\mbox{ for all } f\in B_1.
$$
When $g(z)=z$, the operator $T_zf(z)=\int^{z}_{0}f(w)dw$ becomes the simplest Volterra operator. An integral operator related to $T_g$, denoted by $I_g$, is defined by
$$(I_gf)(z)=\int^{z}_{0}f'(w)g(w)\,dw, ~\mbox{ for all } f\in B_1.
$$

The Volterra type operator $T_g$ was originally studied by Pommerenke \cite{Pom}. Later, a series  of articles appeared on the study of Volterra type integration
operators on classical spaces of analytic functions, such as Hardy space, Bergman space and Dirichlet type spaces etc. For more details, please refer to \cite{A,AS1,AS2,GP,GD}.
In \cite{ZCBP}, \v{C}u\v{c}kovi\'{c} and Paudyal   describe the lattice of the closed invariant subspaces of Volterra type operators.
 Lin et al. \cite{LQ1}  generalized some of the works of \cite{ZCBP} to the general case when $1\leq p<\infty$, and obtained the boundedness of the Volterra type operator $T_g$ and $I_g$ on derivative Hardy space $S_p(\mathbb{D})$. And then they also considered  strict singularity of Volterra type operators on Hardy spaces in \cite{LQ2}. Meanwhile, Lin \cite{L} characterized the boundedness and compactness of the Volterra type operators between Bloch type spaces and weighted Banach spaces. In \cite{MJ}, Miihkinen et al. completely characterize the boundedness of the Volterra type operators  acting from the weighted Bergman spaces  to the Hardy spaces  of the unit ball.

In this paper, we mainly study the  operator-theoretic properties  in minimal M\"{o}bius invariant space $B_1$.  structure of this article is as follows. In Section \ref{sec2},
we discuss the boundedness of the Volterra operator  on $B_1$. In Section \ref{LPX3-sec3}, it is shown whether the integral operator belongs to Deddens algebras.
In Section \ref{sec4}, we will be concerned with the essential norms of integral operators on $B_1$. Section \ref{sec5} is devoted to the study of the spectrum of integral operators on $B_1$.

Throughout this paper, we use the following convention. For two non-negative functions $F$ and $G$ defined on some function space $X,$ we write $F \lesssim G$ if $F(f) \leq C\cdot G(f)$ for all $f \in X$ and for some positive constant $C$ which is independent of $F$ and $G.$
Denote by $F\approx G$ whenever $F\lesssim G \lesssim F$.

\section{Volterra type operators on $B_1$}\label{sec2}

\begin{lemma}\label{le1}
If $f\in H^1$ and $f(z)= \sum^{\infty}_{n=0}a_nz^n$, then
$$\sum^{\infty}_{n=0}\frac{|a_n|}{n+1}\leq\pi\|f\|_{H^1}.
$$
\end{lemma}

The following lemma is a classic exercise in mathematical analysis, but it might be worth to give a brief details of the proof for completeness.

\begin{lemma}\label{le2}
Suppose that $f(x)$ is a continuously increasing function on $[a,b]$. Then
$$\int^{b}_{a}xf(x)dx\geq\frac{a+b}{2}\int^{b}_{a}f(x)\,dx.
$$
\end{lemma}
\begin{proof}
As $f(x)$ is monotonically increasing on $[a,b]$,  the ``integral mean value theorem" shows that there exists a $\xi\in [a,b]$ such that
\begin{align*}
\int^{b}_{a}\left(x-\frac{a+b}{2}\right)f(x)\,dx=&f(a)\int^{\xi}_{a}\left(x-\frac{a+b}{2}\right)dx+f(b)\int^{b}_{\xi}\left(x-\frac{a+b}{2}\right)dx\\
=&\frac{1}{2}(f(b)-f(a))(a\xi+b\xi-\xi^2-ab)\\
=&\frac{1}{2}(f(b)-f(a))(b-\xi)(\xi-a)\geq0.
\end{align*}
This completes the proof of the lemma.
\end{proof}

\begin{lemma}\label{le3}
If $f\in \mathcal{D}^1$, then
$\|f\|_{H^1}\leq\|f\|_{\mathcal{D}^1}.$
\end{lemma}

\begin{proof}
Let $f\in \mathcal{D}^1.$ Then we can see that
$$|f(e^{i\theta})|-|f(0)|\leq|f(e^{i\theta})-f(0)|\leq\left|\int^1_0f'(re^{i\theta}) \,dr\right|\leq\int^1_0|f'(re^{i\theta})| \,dr.
$$
Thus,
\begin{align*}
\|f\|_{H_1}=&\frac{1}{2\pi}\int^{2\pi}_{0}|f(e^{i\theta})| \,d\theta\\
\leq&\frac{1}{2\pi}\int^{2\pi}_{0}\left[|f(0)|+\int^{1}_{0}|f'(re^{i\theta})| \,dr\right]d\theta\\
\leq&|f(0)|+\frac{1}{2\pi}\int^{2\pi}_{0}\int^{1}_{0}|f'(re^{i\theta})| \,dr \,d\theta.
\end{align*}
On the other hand,
\begin{align*}
\|f\|_{\mathcal{D}^1}=|f(0)|+\int_{\mathbb{D}}|f'(z)| \,dA(z)
=&|f(0)|+\frac{1}{\pi}\int^{2\pi}_{0}\int^{1}_{0}r|f'(re^{i\theta})| \,dr \,d\theta.
\end{align*}
By Hardy's convexity theorem (see, \cite{MJ1,XZ}), we find that
$$F(r)=\frac{1}{2\pi}\int^{2\pi}_{0}|f'(re^{i\theta})|\,d\theta,\, ~0<r<1,
$$
is a nondecreasing function of $r$.
It follows from Lemma \ref{le2} that
$$\int^{1}_{0}F(r)dr\leq2\int^{1}_{0}rF(r) \,dr.
$$
This shows that $\|f\|_{H^1}\leq\|f\|_{\mathcal{D}^1}$.
\end{proof}

\begin{remark}\label{re1}
In Lemma \ref{le3}, set $f(z)=z$. Then we see that $1=\|z\|_{H^1}\leq\|z\|_{\mathcal{D}^1}=1$ showing that the norm estimate is sharp. This improves the previous conclusion, namely,  $\|f\|_{H^1}\leq2\|f\|_{\mathcal{D}^1}$  from the works of Girela and Merch{\'a}n \cite{GM}.
\end{remark}

\begin{lemma}\label{le4}
If $f\in B_1$, then $\|f\|_{\infty}\leq\pi\|f\|_{B_1}$ and $B_1\subset H^\infty$.
\end{lemma}

\begin{proof}
Assume that $f\in B_1$ and write $f(z)=\sum^{\infty}_{k=0}a_kz^k$. By Lemmas \ref{le1} and \ref{le3}, we have
\begin{align*}
|f(z)|=\left|\sum^{\infty}_{k=0}a_kz^k\right|\leq\sum^{\infty}_{k=0}|a_kz^k|
\leq&\sum^{\infty}_{k=0}|a_k| \leq \pi\|f'\|_{H^1}+|f(0)| \leq \pi\|f'\|_{\mathcal{D}^1}+|f(0)|\leq\pi\|f\|_{B_1},
\end{align*}
for all $z\in \mathbb{D}$.
%
Hence, we obtain that $\|f\|_{\infty}\leq\pi\|f\|_{B_1}$.
\end{proof}

\begin{remark}\label{re2}
In \cite[Theorem 1]{LQ1}, they  obtained that $\|f\|_{\infty}\leq\pi\|f\|_{S^1}$ for each $f\in S^1$,  where the space $S^1$ is defined by $S^1=\{f\in H^1:f'\in H^1\}$. The norm  on $S^1$ is given by
$$\|f\|_{S^1}=|f(0)|+\|f'\|_{H^1}.
$$
Moreover, we obtain  the following  norm estimate
$$\|f\|_{\infty}\leq\pi\|f\|_{S^1}\leq\pi\|f\|_{B_1}\ \ \mbox{for all $f\in B_1$.}
$$
\end{remark}

In the following, we discuss the boundedness of $T_g$ and $I_g$ on minimal M\"{o}bius invariant space.

\begin{theorem}\label{th1}
The operator $T_g$ is  bounded on $B_1$ if and only if $g\in B_1$. Moreover,
$$\|g-g(0)\|\leq \|T_g\|\leq (1+\pi) \|g-g(0)\|_{B_1}.
$$
\end{theorem}

\begin{proof}
Let $f\in B_{1}$. By H\"{o}lder's inequality and Lemma \ref{le3}, we have
\begin{align*}
\int_{\mathbb{D}}|f'(z)g'(z)| \,dA(z)&\leq\left(\int_{\mathbb{D}}|f'(z)|^2 \, dA(z)\right)^{\frac{1}{2}}\left(\int_{\mathbb{D}}|g'(z)|^2 \,dA(z)\right)^{\frac{1}{2}}
=\|f'\|_{A^2}\|g'\|_{A^2}\\
&\leq\|f'\|_{H^1}\|g'\|_{H^1}  \\
&\leq\|f'\|_{\mathcal{D}^1}\|g'\|_{\mathcal{D}^1} \\
&\leq \|f\|_{B_1}\|g-g(0)\|_{B_1}.
\end{align*}
Hence, we get
\begin{align*}
\|T_gf\|_{B_1}=&\|fg'\|_{\mathcal{D}^1(\mathbb{D})}
= |f(0)g'(0)|+\int_{\mathbb{D}}|f'(z)g'(z)| \,dA(z)+\int_{\mathbb{D}}|f(z)g''(z)| \,dA(z)\\
\leq& \|f\|_{B_1}\|g-g(0)\|_{B_1}+\|f\|_{\infty}\|g-g(0)\|_{B_1}\\
\leq& (1+\pi) \|f\|_{B_1}\|g-g(0)\|_{B_1}
\end{align*}
showing that $T_g$ is a bounded operator on $B_1$.

Conversely, assume that $T_g$ is a bounded operator on $B_1$ and let $f=1$. Then we obtain
$$\|T_g\|\geq\|T_g1\|_{B_1}\geq \|g-g(0)\|_{B_1}
$$
 which gives that $g\in B_1$. Thus,
$$\|g-g(0)\|\leq \|T_g\|\leq (1+\pi) \|g-g(0)\|_{B_1}.
$$
\end{proof}

For $0<p<\infty$, $-2<q<\infty$, and $0\leq s<\infty$, we define the general family of function spaces $F(p,q,s)$ as the set of  all analytic functions $f$ in $\mathbb{D}$ such that
$$\|f\|^{p}_{p,q,s}=|f(0)|+\sup_{a\in \mathbb{D}}\int_{\mathbb{D}}|f'(z)|^{p}(1-|z|^2)^{q}g^{s}(z,a) \,dA(z)<\infty,
$$
where $g(z,a)=\log\frac{1}{|\sigma_a(z)|}$. These spaces were introduced by Zhao in \cite{ZR}. In 2003, R\"{a}tty\"{a} provided the following $n$-th derivation characterization of functions in  spaces $F(p,q,s)$.

\begin{lemma}\cite[Theorem 3.2]{RJ}\label{le5}
Let $f$ be an analytic function on $\mathbb{D}$ and let $0<p<\infty,-2<q<\infty$ and $0\leq s<\infty$. Let $n\in \mathbb{N}$ and $q+s>-1$; or $n=0$ and $q+s-p>-1$. Then $f\in F(p,q,s)$
if and only if
$$\sup_{a\in \mathbb{D}}\int_{\mathbb{D}}|f^{(n)}(z)|^p(1-|z|^2)^{np-p+q}(1-|\varphi_a(z)|^2)^s \,dA(z)<\infty.
$$
\end{lemma}

For $p>0$, the space $Z_p$ consists of  all analytic functions $f$ in $\mathbb{D}$ such that
$$\|f\|_{Z_p}=|f(0)|+\sup_{a\in \mathbb{D}}\int_{\mathbb{D}}|(f\circ\sigma_a(z))'|(1-|z|^2)^{p-1} \,dA(z)<\infty.
$$
It is clear that $Z_1=F(1,-1,1)$. For more results about $Z_{p}$ space, see \cite{LLZ} and \cite{Zhu1}.

\begin{theorem}\label{th2}
The operator $I_g$ is bounded on $B_1$ if and only if $g\in Z_{1}\cap H^{\infty}$.
\end{theorem}

\begin{proof}
Assume that $g \in Z_{1}\cap H^{\infty}$.
Using Lemma $6$ of \cite{LLZ}, we obtain that
$$\int_{\mathbb{D}}|f'(w)|\cdot|g'(w)| \,dA(w)\lesssim \|g\|_{Z_{1}} \left(|f'(0)|+\int_{\mathbb{D}}|f''(w)| \,dA(w)\right).
$$
 Thus, we have
\begin{align*}
\|(I_gf)(z)\|_{B_1}=&\|f'g\|_{\mathcal{D}^1}\\
=& |f'(0)g(0)|+ \int_{\mathbb{D}}|f''(w)|\cdot|g(w)| \,dA(w)+\int_{\mathbb{D}}|f'(w)|\cdot|g'(w)| \,dA(w)\\
\lesssim & |f'(0)g(0)|+ \|g\|_{\infty}\int_{\mathbb{D}}|f''(w)| \,dA(w)+\|g\|_{Z_{1}}\left(|f'(0)|+\int_{\mathbb{D}}|f''(w)| \,dA(w)\right)\\
\lesssim&(\|g\|_{\infty}+\|g\|_{Z_{1}})\|f\|_{B_1}
\end{align*}
for all $f\in B_1$. This implies that $I_g$ is a bounded operator on $B_1$.

Conversely, suppose that $I_g$ is a bounded operator on $B_1$. For each $a\in \mathbb{D}$,
let $f(z)=\sigma_{a}(z)$. Then  $\|f\|_{B_1}\lesssim  1$ and
\begin{align*}
\|I_g\| \gtrsim &\|I_gf\|_{B_1}
=\|g(w)f'(w)\|_{\mathcal{D}^1}\\
\geq &\|g(w)f'(w)\|_{H^1}
=\int^{2\pi}_{0}|g(e^{i\theta})|\cdot|\sigma'_{a}(e^{i\theta})| \,d\theta
=\int^{2\pi}_{0}|g(\sigma_{a}(e^{i\theta}))| \,d\theta\\
\geq&  |g(\sigma_{a}(0))|=|g(a)|
\end{align*}
which shows that $\|I_g\|\gtrsim \|g\|_{\infty}$.
Hence, $g\in H^{\infty}$. Moreover,
\begin{align*}
\|I_g\| & \gtrsim \|I_g\sigma_a\|_{B_1}\\
&\geq\int_{\mathbb{D}}|\sigma'_{a}(w)|\cdot|g'(w)| \,dA(w)-\int_{\mathbb{D}}|\sigma''_{a}(w)|\cdot|g(w)| \,dA(w)\\
&\geq\int_{\mathbb{D}}|\sigma'_{a}(w)|\cdot|g'(w)| \,dA(w)-\|g\|_{\infty}\|\sigma_a\|_{B_1},
\end{align*}
from which it follows that
$$\|I_g\|+\|g\|_{\infty}\|\sigma_a\|_{B_1} \gtrsim  \int_{\mathbb{D}}|\sigma'_{a}(w)|\cdot|g'(w)| \,dA(w).
$$
This implies that $g\in Z_1$. Therefore, $g\in Z_1\cap H^{\infty}$. This completes the proof of the theorem.
\end{proof}

\begin{remark}\label{re3}
Note  that $B_1\subset Z_1\cap H^{\infty}.$  By Lemma \ref{le4}, we know that $f\in H^{\infty}$ whenever $f\in B_1$.
Using the $2$-nd derivation characterization of functions in  $Z_1$, we have
\begin{align*}
\|f\|_{Z_1}\approx |f(0)|+\sup_{a\in \mathbb{D}}\int_{\mathbb{D}}|f''(z)|(1-|\sigma_a(z)|^2) \,dA(z)\leq |f(0)|+\int_{\mathbb{D}}|f''(z)| \,dA(z)\leq\|f\|_{B_1}.
\end{align*}
Then $f\in Z_1$ and thus, $B_1\subset Z_1\cap H^{\infty}$.
\end{remark}

Now we define the space $B^0_1$ as
$$B^0_1=\{f\in B_1:f(0)=0\}.
$$
The following theorem  gives the connection between the Volterra operator $T_z$ on $\mathcal{D}^1$ and the multiplication operator $M_z$ on $B^0_1$.

\begin{theorem}\label{th3}
The Volterra type operator $T_{z}:\mathcal{D}^1\rightarrow  B^0_1$ is bounded and invertible with $T^{-1}_z= D .$
\end{theorem}

\begin{proof}
First, we show that $T_z(\mathcal{D}^1)=B^0_1.$ For  $f\in \mathcal{D}^1$, we consider
$$F(z):=(T_zf)(z)=\int^{z}_{0}f(w)\,dw.
$$
 Clearly, $F'=f\in \mathcal{D}^1$ so that  $F\in B^0_1$ and  $T_{z}(\mathcal{D}^1)\subseteq B^0_1.$

Conversely, for each $F\in B^0_1$, we have $F' \in \mathcal{D}^1$.  Then
$$(T_z(F'(z)))=\int^{z}_{0}F'(w)\,dw=F(z)-F(0)=F(z).
$$
Then $B^0_1 \subseteq T_{z}(\mathcal{D}^1)$. This implies that   $T_{z}(\mathcal{D}^1)=B^0_1$.

Secondly, we show that $T_z:\mathcal{D}^1 \rightarrow B^0_1$ is a bounded isomorphism, and its inverse $T^{-1}_{z}= D.$  Recall from the above
discussion that $(T_z(F'(z)))=F(z)$ for  $F\in B^0_1$. Then, for each $f\in\mathcal{D}^1$, we have that $(D(T_zf))(z)=f(z)$. This implies that $T_z$ is a bijection from $\mathcal{D}^1$ onto $B^0_1$. Since $T_z$ is linear and $T_z$ is an isomorphism from $\mathcal{D}^1$ onto $B^0_1$.

Finally, we need to prove that $T_z$ is a bounded operator on $\mathcal{D}^1$. For $f\in \mathcal{D}^1$, we have
$$\|T_zf\|_{B_1}\lesssim ( |(T_zf)(0)|+\|f\|_{\mathcal{D}^1})\lesssim \|f\|_{\mathcal{D}^1}.
$$
 Therefore, $T_z$ is a bounded isomorphism from $\mathcal{D}^1$ onto $B^0_1$.
\end{proof}

Let us introduce an addition operator $P$ defined by
$$(Pf)(z)=(M_zf)(z)+(T_zf)(z)\ \mbox{ for $f\in H(\mathbb{D})$ and $z\in\mathbb{D}$.}
$$

\begin{theorem}\label{th4}
Let $T_{z}:\mathcal{D}^1\rightarrow B^0_1$ and $M_z:B^0_1\rightarrow \mathcal{D}^1$ be the Volterra type operator and the multiplication operator, respectively. Then $P$ is an operator on $\mathcal{D}^1$ with
$P=T^{-1}_{z}M_zT_z.$
\end{theorem}

\begin{proof}
Let $f\in \mathcal{D}^1$ and $F=T_zf$. Then we get
\begin{align*}
(Pf)(z)=&(M_zf)(z)+(T_zf)(z)
=zf(z)+F(z)
=D(zF(z))
=(T_z^{-1}M_zT_zf)(z),
\end{align*}
which shows that $P=T^{-1}_{z}M_zT_z$.
\end{proof}

\begin{theorem}\label{th5}
If $f,g\in B_1$, then $\|fg\|_{B_1}\leq(2\pi+2)\|f\|_{B_1}\|g\|_{B_1}$.
\end{theorem}

\begin{proof}
For $f,g\in B_1$, we get
\begin{align*}
\|f g\|_{B_1}=&|f(0)g(0)|+|f'(0)g'(0)|\\
&+\int_{\mathbb{D}}|f''(w)g(w)+2f'(w)g'(w)+f(w)g''(w)| \,dA(w)\\
\leq& (|f(0)|+|f'(0)|)(|g(0)|+|g'(0)|)+\int_{\mathbb{D}}|f''(w)g(w)| \,dA(w)\\
&+2\int_{\mathbb{D}}|f'(w)g'(w)| \,dA(w)+\int_{\mathbb{D}}|f(w)g''(w)| \,dA(w)\\
\leq&\|g\|_{B_1}(|f(0)|+|f'(0)|) +\|g\|_{\infty} (\|f\|_{B_1}-|f(0)|-|f'(0)|)\\
&+2\|f\|_{B_1}\|g\|_{B_1}+\|f\|_{\infty}\int_{\mathbb{D}}|g''(w)| \,dA(w)\\
\leq&\|g\|_{B_1}(|f(0)|+|f'(0)|)+\pi\|g\|_{B_1}(\|f\|_{B_1}-|f(0)|-|f'(0)|)\\
&+2\|f\|_{B_1}\|g\|_{B_1}+\pi\|f\|_{B_1}\|g\|_{B_{1}}\\
\leq&(2\pi+2)\|f\|_{B_1}\|g\|_{B_1}
\end{align*}
and the proof is complete.
\end{proof}

\begin{remark}
In \cite[Theorem10]{AJ}, Arazy et al. obtained $\|fg\|\leq7\|f\|\|g\|$ for  $f,g\in B_1$, in which they defined the norm of $f\in B_1$ as
$$\|f\|=\inf\left\{\sum^{\infty}_{k=1}|c_{k}|:f(z)=\sum^{\infty}_{k=1}c_k\sigma_{a_k}(z)\right\}.
$$
Inspired by their work, we derived Theorem \ref{th5} and we are not sure whether the constant $2\pi+2$ in Theorem \ref{th5} is optimal or not.
\end{remark}

\section{Deddens algebras}\label{LPX3-sec3}
Let $\mathcal{L}(X)$ denote the algebra of all bounded linear operators on a complex Banach space $X$. A nontrivial invariant subspace of an operator $A\in\mathcal{L}(X)$ is, by definition, a closed subspace $M$ of $X$ such that $M\neq \{0\}$, $M\neq X$, and $Ax\in M$ for every $x\in M$; or, briefly, $A(M)\subset M$.

Let $A\in \mathcal{L}(X)$. The operator $T$ is said to belong to Deddens algebra $\mathcal{D}_A$ if there exists $M=M(T)>0$ such that
$$\|A^nTf\|\leq M\|A^nf\|
$$
for each $n\in \mathbb{N}$ and $f\in X$.

The study of Deddens algebra was originally introduced by Deddens \cite{D}, where he assumed that $A$ is an invertible operator and $\sup_{n\in \mathbb{N}}\|A^nTA^{-n}\|<\infty$. Later, 
it received the attention of many, see \cite{D,T,K,P,PE,PS,LM,SD,PSD}. Recently,
Petrovic\cite{PSD} studied the Deddens algebra associated to compact composition operators $C_{\varphi}$ on Hardy spaces, where $A$ is not necessarily be invertible. And they have demonstrated that the operators $M_g$ and $T_z$ belong to the Deddens algebra $\mathcal{D}_{C_{\varphi}}$. It is worth to point out that compact operator on $H^2$ is not invertible.

Let us begin to present the boundedness of composition operator and multiplication operator on $B_1$.

\begin{lemma}\cite{W}\label{le6}
Let $\varphi$ be an  analytic  self-map of $\mathbb{D}$. Then the composition operator $C_{\varphi}$ is bounded on $B_1$ if and only if $$\sup_{a\in\mathbb{D}}\|C_{\varphi}\sigma_a\|_{B_1}<\infty.
$$
\end{lemma}

In particular, we give another description of sufficiency condition for the boundedness of the composition operator in the following theorem.

\begin{theorem}\label{th6}
Let $\varphi$ be an  analytic  self-map of $\mathbb{D}$ such that $\varphi(0)=0$. Then the composition operator $C_{\varphi}$ is bounded on $B_1$  whenever $\varphi'\in Z_1\cap H^{\infty}$.
\end{theorem}

\begin{proof}
Suppose that $\varphi'\in Z_1\cap H^{\infty}$. Then we get
\begin{align*}
\|C_{\varphi}f\|_{B_1}=&\|f(\varphi(z))\|_{B_1}\\
=&|f(\varphi(0))|+|f'(\varphi(0))\varphi'(0)|+\int_{\mathbb{D}}|f''(\varphi(w))\cdot(\varphi'(w))^2
+f'(\varphi(w))\cdot\varphi''(w)| \,dA(w)\\
\leq &|f(\varphi(0))|+|f'(\varphi(0))\varphi'(0)|+\int_{\mathbb{D}}|f''(\varphi(w))\cdot(\varphi'(w))^2| \,dA(w)
\\
&+\int_{\mathbb{D}}|f'(\varphi(w))\cdot\varphi''(w)| \,dA(w)\\
\lesssim  &\|f\|_{\infty}+\|\varphi'\|_{\infty} \|f\|_{B_{1}}+\|\varphi'\|^{2}_{\infty} \|f\|_{B_{1}}
+\|\varphi\|_{Z_{1}} \|\varphi'\|_{\infty} \|f\|_{B_{1}} \\
\lesssim  & (\|\varphi'\|^{2}_{\infty}+\|\varphi\|_{Z_{1}} \|\varphi'\|_{\infty}+ \|\varphi'\|_{\infty} +1)\|f\|_{B_1},
\end{align*}
which implies that  $C_{\varphi}$ is bounded on $B_1$ if  $\varphi'\in Z_1\cap H^{\infty}$.
\end{proof}

\begin{theorem}\label{th7}
Suppose that $M_g$ is multiplication operator on $B_1$. Then $M_g$ is bounded if and only if $g\in B_1$.
\end{theorem}

\begin{proof}
For any $f\in B_1$, if $g\in B_1$, we have $M_g$ is bounded on $B_1,$ by Theorem \ref{th5}.

Conversely, let $M_g$ be a bounded operator on $B_1.$ Then with $f=1$, we get $\|M_g\|\geq\|M_g1\|_{B_1}=\|g\|_{B_1}$ which implies that $g\in B_1$.
\end{proof}

In the following theorem, we will consider the algebra $\mathcal{D}_{C_{\varphi}}$, in which the operator $C_{\varphi}$ is a bounded composition operator.
For $n\in \mathbb{N}$, it clear that $C^{n}_{\varphi}f=f\circ\varphi\circ\cdots\circ\varphi$. For simplicity of  the notation, we write
$\varphi_{n}$ instead of  $\varphi\circ\cdots\circ\varphi$.

\begin{theorem}\label{th8}
Let $g\in B_1$ and $\varphi$ be an analytic self-map of $\mathbb{D}$ with $\varphi(0)=0$ such that $C_{\varphi}$ is bounded on $B_1$ . Then the operators $M_g$, $T_g$ and $I_g$ belong to the Deddens algebra $\mathcal{D}_{C_{\varphi}}$.
\end{theorem}
\begin{proof}
 For each $n\in \mathbb{N}$, we see that
\begin{align*}
C^n_{\varphi}M_gf=&C^{n}_{\varphi}(gf)
=(g\circ\varphi_{n})(f\circ\varphi_{n})
=M_{g\circ\varphi_n}C^{n}_{\varphi}f.
\end{align*}
Since $\varphi_{n}(\mathbb{D})\subset \mathbb{D}$, it follows that
$$\|M_{g\circ\varphi_n}f\|_{B_1}=\|(g\circ\varphi_n)f\|_{B_1}\lesssim\|g\|_{B_1}\|f\|_{B_1}
$$
and therefore,
$$\|C^n_{\varphi}M_gf\|_{B_1}=\|M_{g\circ\varphi_n}C^{n}_{\varphi}f\|_{B_1}\lesssim\|g\|_{B_1}\|C^{n}_{\varphi}f\|_{B_1},
$$
where $f\in B_{1}(\mathbb{D})$. This implies that $M_g\in \mathcal{D}_{C_{\varphi}}$.

Next, we have
$$C^n_{\varphi}T_gf(z)=C^n_{\varphi}\left(\int^{z}_{0}f(w)g'(w)dw\right)=\int^{\varphi_{n}(z)}_{0}f(w)g'(w) \,dw.
$$
Since $\varphi$ is an analytic self-map of $\mathbb{D}$ satisfying $\varphi(0)=0$, we have $\varphi_{n}(0)=0$ and therefore,
\begin{align*}
T_{g\circ\varphi_n}C^{n}_{\varphi}f(z)=&T_{g\circ\varphi_n}\left(f(\varphi_n(z))\right)\\
=&\int^{z}_{0}f(\varphi_{n}(w))g'(\varphi_{n}(w))\varphi_{n}'(w) \,dw\\
=&\int^{\varphi_{n}(z)}_{0}f(w)g'(w) \,dw,
\end{align*}
where $f\in B_1$.
It shows that $C^n_{\varphi}T_g=T_{g\circ\varphi_n}C^{n}_{\varphi}$.
By Theorem \ref{th1}, we have
\begin{align*}
\|C^{n}_{\varphi}T_gf(z)\|_{B_1}=&\|T_{g\circ\varphi_n}C^{n}_{\varphi}f(z)\|_{B_1}
\lesssim \|g\circ\varphi_n\|_{B_1}\|C^{n}_{\varphi}f(z)\|_{B_1}
\lesssim \|g\|_{B_1}\|C^{n}_{\varphi}f(z)\|_{B_1}.
\end{align*}
which gives $T_g\in \mathcal{D}_{C_{\varphi}}$.

  Finally, we have
$$C^n_{\varphi}I_gf(z)=C^{n}_{\varphi}\left(\int^{z}_{0}f'(w)g(w) \,dw\right)=\int^{\varphi_{n}(z)}_{0}f'(w)g(w) \,dw
$$
and
\begin{align*}
I_{g\circ\varphi_{n}}C^n_{\varphi}f(z)=&I_{g\circ\varphi_{n}}(f(\varphi_{n}(z)))\\
=&\int^{z}_{0}f'(\varphi_{n}(w))g(\varphi_{n}(w))\varphi_{n}'(w) \,dw\\
=&\int^{\varphi_n(z)}_{0}f'(w)g(w) \,dw,
\end{align*}
where $f\in B_1$.
Therefore, $C^n_{\varphi}I_g=I_{g\circ\varphi_{n}}C^n_{\varphi}$.
By Theorem \ref{th2} and Remark \ref{re3}, we find that
$$\|C^n_{\varphi}I_gf(z)\|_{B_1}=\|I_{g\circ\varphi_{n}}C^n_{\varphi}f(z)\|_{B_1}
\lesssim\|g\circ\varphi_{n}\|_{B_1}\|C^n_{\varphi}f(z)\|_{B_1}\lesssim\|g\|_{B_1}\|C^n_{\varphi}f(z)\|_{B_1}
$$
for all $f\in B_1$.
We thus deduce that $I_g\in \mathcal{D}_{C_{\varphi}}$.
\end{proof}


\section{Essential norms of Volterra type operators on $B_{1}$}\label{sec4}
Suppose $X$ is a Banach space and $T$ is a bounded linear operator on $X$. The essential norm of $T$ is defined 
to be
$$\|T\|_{e}=\inf \{\|T-K\|:K ~\text{is a compact operator on } B_1\}.
$$
Obviously, the essential norm of $T$ is 0 if and only if $T$ is compact. For more results, we invite the reader to refer to \cite{SE,LLL}.
In this section, we  characterize the essential  norm of linear operator on $B_1$, which generalizes the conclusion of Liu et al.\cite{LLX}

\begin{theorem}\label{th9}
Every bounded operator $T_g$ on $B_1$ is compact.
\end{theorem}

\begin{proof}
By definition, we know  $\|T_{g}\|_{e}\geq0$.

Next we show that $\|T_{g}\|_{e}\leq0$. To do this we define the following operators:
$$T_{g_r}f:=\int^z_0rf(w)g_r'(w) \,dw,
$$
where $g_r(z)=g(rz)$ and $r\in(0,1)$. It is easy to see that $T_{g_r}$  is  a compact operator on $B_1$ for $g\in B_1$.
In fact,   if $T_{g}$ is  a bounded operator in $B_1$, then $g\in B_1$.
Suppose that $\{ f_{n}\}_{n=1}^{\infty} \subset B_1$ with $\|f_{n}\|_{B_1}\leq 1$, and $f_{n}\rightarrow 0$ uniformly on compact subsets of $\mathbb{D}$.
Then
\begin{align*}
\|T_{g_r}f_n\|_{B_1}=&\|f_{n}(rz)g'(rz)\|_{\mathcal{D}^1}\\
\lesssim&\int_{\mathbb{D}}|f'_{n}(rz)g'(rz)| \,dA(z)+\int_{\mathbb{D}}|f_{n}(rz)g''(rz)| \,dA(z)\\
=&\int_{\mathbb{D}_{\delta_1}}|f'_{n}(rz)g'(rz)| \,dA(z)+\int_{\mathbb{D}\backslash \mathbb{D}_{\delta_1}}|f'_{n}(rz)g'(rz)| \,dA(z)\\
&+\int_{\mathbb{D}_{\delta_1}}|f_{n}(rz)g''(rz)| \,dA(z)+\int_{\mathbb{D}\backslash \mathbb{D}_{\delta_1}}|f_{n}(rz)g''(rz)| \,dA(z),
\end{align*}
where $\delta_1<1,$ and $\mathbb{D}_{\delta_1}= \{z: |z|<\delta_1\}$ are compact subsets of $\mathbb{D}$.

Note that
$$\int_{\mathbb{D}}|f'_{n}(rz)g'(rz)| \,dA(z)\lesssim \|f_{n}\|_{B_1} \|g\|_{B_1}   ~\text{ and }~    \int_{\mathbb{D}}|f_{n}(rz)g''(rz)| \,dA(z)\lesssim \|f_{n}\|_{B_1} \|g\|_{B_1} .
$$
Using the theorem of absolute continuity of Lebegue measure, we conclude that
$$\int_{\mathbb{D}\backslash \mathbb{D}_{\delta_1}}|f'_{n}(rz)g'(rz)| \,dA(z) <\epsilon  ~\text{ and }~    \int_{\mathbb{D}\backslash \mathbb{D}_{\delta_1}}|f_{n}(rz)g''(rz)| \,dA(z) <\epsilon .
$$

On the other hand, using a basic result of complex analysis (see [p. 151 of \cite{CJ}])
 we know that if $f_{n}\rightarrow 0,$ then $f_{n}' \rightarrow 0$ uniformly on  compact subsets of $\mathbb{D}$. Consequently,
$$\int_{\mathbb{D}_{\delta_1}}|f'_{n}(rz)g'(rz)| \,dA(z) ~\text{ and }~ \int_{\mathbb{D}_{\delta_1}}|f_{n}(rz)g''(rz)| \,dA(z)
$$
converge to $0$ when $n\rightarrow \infty$. This implies that $T_{g,r}$ is compact.

Meanwhile, we have
$$\|T_{g}\|_{e}\leq\|T_{g}-T_{g_r}\|=\|T_{g-g_r}\|\lesssim\|g-g_r\|_{B_1}
$$
for $r\in(0,1)$. Similarly with the above computation, we have
$$\lim_{r\rightarrow 1^{-}}\|g-g_r\|_{B_1}=0
$$
so that $\|T_{g}\|_{e}\leq0$. Hence we deduce that $\|T_{g}\|_{e}=0$, which completes the proof.
\end{proof}


\begin{theorem}\label{th10}
If $I_{g}$ is bounded operator on $B_1$, then $\|I_{g}\|_{e} \approx \|g\|_{\infty}.$
\end{theorem}

\begin{proof}
From the proof of Theorem \ref{th2}, we have
$$\|I_g\|_{e}=\inf\|I_g-K\|\leq\|I_g\|\leq C\|g\|_{\infty},
$$
where $C$ is a positive constant.

We next show that $\|I_g\|_{e}\geq \|g\|_{\infty}$.
Choose $a_n\in \mathbb{D}$ such that $|a_n|\rightarrow1$ as $n\rightarrow\infty$. Let $f_n(z)=\sigma_{a_n}(z)-a_n$.
It is obvious that $\|f_n\|_{B_1}=1$. Since $\{f_n\}$ converges to zero uniformly on compact subsets of $\mathbb{D}$, for every compact operator $K$ on $B_1$, we obtain $\|Kf_n\|_{B_1}\rightarrow0$ as $n\rightarrow\infty$. Therefore,
\begin{align*}
\|I_g-K\|&\geq\lim_{n\rightarrow\infty}\sup\|(I_g-K)f_n\|_{B_1}\\
&\geq\lim_{n\rightarrow\infty}\sup(\|I_gf_n\|_{B_1}-\|Kf_n\|_{B_1})\\
&=\lim_{n\rightarrow\infty}\sup\|I_gf_n\|_{B_1}.
\end{align*}
Similar to the proof of Theorem \ref{th2}, we get
\begin{align*}
\|I_gf_n\|_{B_1}=\|g(z)f^{'}_n(z)\|_{\mathcal{D}^{1}}\geq|g(\sigma_{a_n}(0))|=|g(a_{n})|.
\end{align*}
As the choice of the sequence $\{a_n\}\subset\mathbb{D}$ is arbitrary, we have $\|I_g\|_{e}\geq\|g\|_{\infty}$, which is completes the proof.
\end{proof}

\begin{theorem}\label{th11}
Every bounded operator $M_g$ on $B_1$ is compact.
\end{theorem}

\begin{proof}
The proof is similar to Theorem \ref{th9}, so we omit its details.
\end{proof}

\section{Spectrum of Volterra type operators on $B_{1}(\mathbb{D})$}\label{sec5}
The spectrum of integral operators in different spaces has attracted the attention of many scholars. The spectrum of integral operator on weighted Bergman space are characterized by Aleman and Constantin\cite{AC}.
Later, Constantin \cite{CO} obtained the spectrum of Volterra type operators on Fock spaces.
Mengestie\cite{MT} studied of the spectrum of Volterra type operators on Fock-Sobolev spaces. Mengestie\cite{MT2} also obtained the spectrum of $T_g$ in
terms of a closed disk of radius twice the coefficient of the highest degree term in a polynomial expansion of $g$. For more results, see\cite{MB,BJ}.
Recently, Lin et al. described the spectrum of the multiplication operator and Volterra type operator $I_g$ in \cite{LQ2}, respectively.
Inspired by the above results, it is natural to discuss the spectrum of the multiplication operator and the Volterra type operators in $B_{1}$.

\begin{theorem}\label{th12}
Suppose that $M_g$ is bounded on $B_1$. Then we have
$\sigma(M_g)=\overline{g(\mathbb{D})}.$
\end{theorem}

\begin{proof}
 Suppose that $\lambda\notin\sigma(M_g)$. Then  $M_g-\lambda I$ is invertible. As $1 \in B_1,$ there exists an $f\in B_1$ such that $(g(z)-\lambda)f(z)=1$ for all $z\in \mathbb{D}$, which implies that $\lambda\notin g(\mathbb{D})$. Thus, $g(\mathbb{D})\subset\sigma(M_g)$.

For the other way inclusion, we let $\lambda\notin \overline{g(\mathbb{D})}$. Then we can choose a $t>0$ such that $|g(z)-\lambda|>t$ for all $z\in \mathbb{D}$. This  shows that
$h=(g-\lambda)^{-1}$ is a bounded analytic function on $\mathbb{D}$. For all $g\in B_1$, we get
\begin{align*}
\|h\|_{B_1}\leq& t(|h(0)|+\|h'\|_{\mathcal{D}^{1}})\\
\leq&  \left(\frac{1}{|g(0)-\lambda|}+\|\frac{g'}{(g-\lambda)^2}\|_{\mathcal{D}^{1}}\right)\\
\leq&   \left(\frac{1}{|g(0)-\lambda|}+\frac{g'(0)}{(g(0)-\lambda)^2}+\int_{\mathbb{D}}\left|\frac{g''(z)}{(g(z)-\lambda)^2}-\frac{2(g')^2}{(g(z)-\lambda)^{3}}\right| \,dA(z)\right)\\
\lesssim &    \|g\|_{B_1}+ \|g\|^2_{B_1}.
\end{align*}
 Hence, $h\in B_1$.  Then $M_h$ is bounded on $B_1,$ by Theorem \ref{th7}.
Since $M_h=M_{(g-\lambda)^{-1}}=M_{g-\lambda}^{-1}$, we see that $M_{g-\lambda}$ is invertible and thus, $\lambda\notin \sigma(M_g)$.
Therefore,  $\sigma(M_g)\subset\overline{g(\mathbb{D})}$. Since the spectrum set is closed,  we conclude that $\sigma(M_g)=\overline{g(\mathbb{D})}$.
\end{proof}

\begin{lemma}\label{le8}\cite{R}
Let $T$ be a bounded linear operator on a Banach space $X$, and $T$ be  compact. If $\dim X=\infty$, then $\sigma(T)=\{0\} \cup \{\mbox{eigen values of $T$}\} $.
\end{lemma}

\begin{theorem}\label{th13}
Suppose that $T_g$ is a bounded operator in $B_1$. Then $\sigma(T_g)=\{0\}.$
\end{theorem}

\begin{proof}
Let $T_g$ be a bounded operator on $B_1.$ Then  $T_g$ is compact, by Theorem \ref{th9}. By Lemma \ref{le8}, we obtain $0\in \sigma(T)$.

Next we prove that $T_g$ has no non-zero eigenvalue. Assume that $T_g$ has an eigenvalue $\lambda\neq0$ with eigenvector $f.$
Then
\begin{equation}\label{eq1}
T_{g}f(z)=\int^{z}_{0}f(w)g'(w) \,dw=\lambda f(z).
\end{equation}
Differentiating the equation  (\ref{eq1}) with respect to $z$, we get
$$f(z)g'(z)=\lambda f'(z).
$$
All nonzero solutions of this equation are of the form $f(z)=ce^{\frac{g(z)}{\lambda}}$ for some $c\neq0$. Setting $z=0$ in (\ref{eq1}) shows that $0=\lambda f(0)$,
which contradicts the last relation about $f.$ Therefore, there is no non-zero eigenvalue for $T_g$. From this, we deduce that $\sigma(T_g)=\{0\}$.
\end{proof}

\begin{theorem}\label{th14}
 If $I_g$ is a bounded operator in $B_1$,  then
$$\sigma(I_g)=\{0\}\cup \overline{g(\mathbb{D})}.
$$
\end{theorem}

\begin{proof}
For any constant function $a$, we have
$$(I_ga)(z)=\int_{0}^{z}a'(w)g(w)\,dw=0
$$
which gives $0\in \sigma(I_g)$.

Suppose that $\lambda\in \mathbb{C}\backslash\{0\}$. Note that the equation
$$f-\frac{1}{\lambda}I_gf=h,~\hbox {for} ~h\in B_1,
$$
has a unique analytic solution
$$f(z)=R_{\lambda,g}h(z)=\int^{z}_{0}\frac{h'(\xi)}{1-\frac{1}{\lambda}g(\xi)} \,d\xi+h(0)=I_{(1-\frac{1}{\lambda}g)^{-1}}h(z)+h(0).
$$
(see \cite{CA} for more details). Hence, the resolvent set $\rho(I_g)$ of the bounded operator $I_g$ consists precisely of all points $\lambda\in \mathbb{C}$ for which $R_{\lambda,g}$ is a bounded operator on $B_1$.

If $\lambda\in \mathbb{C}\backslash(\{0\}\cup \overline{g(D)})$, then it is clear that $1-\frac{1}{\lambda}g(z)$ is bounded away from $0$, which implies that  $\frac{1}{1-\frac{1}{\lambda}g(z)} \in H^{\infty}$.   If $I_g$ is a bounded operator in $B_1$, then $g\in H^{\infty}\bigcap Z_{1}$ by Theorem \ref{th2}.
Moreover, it is easy to  show that $ \frac{1}{1-\frac{1}{\lambda}g(z)} \in Z_{1}.$
This implies that the operator $R_{\lambda,g}$ is a bounded operator on $B_1$. It follows that  $\mathbb{C}\backslash(\{0\}\cup \overline{g(\mathbb{D})})\subset \rho(I_g)$. Thus, $\sigma(I_g)\subset(\{0\}\cup \overline{g(\mathbb{D})})$.

On the other hand, if $\lambda\in g(\mathbb{D})$ and $\lambda\neq0$, then $\frac{1}{1-\frac{1}{\lambda}g(\xi)}$ is not bounded, which shows that the operator $R_{\lambda,g}$ is not bounded on $B_1$. Therefore, we obtain that $g(\mathbb{D})\backslash\{0\}\subset\sigma(I_g)$. This together with the fact that  $0\in \sigma(I_g)$ shows that
$$g(\mathbb{D})\cup\{0\}\subset\sigma(I_g)\subset \overline{g(\mathbb{D})}\cup\{0\}.
$$
Since the spectrum $\sigma(I_g)$ is closed, we deduce that $\sigma(I_g)=\overline{g(\mathbb{D})}\cup\{0\}$.
\end{proof}

\end{document}